\documentclass[11pt,a4paper]{article} 
\usepackage{amsfonts}
\usepackage{amsmath}
\usepackage{amssymb}
\usepackage{enumerate}
\usepackage{amsthm} 
\usepackage[final]{graphicx} 
\usepackage{ifthen} 
\usepackage{graphpap} 
\usepackage{xspace}

\numberwithin{equation}{section}
\newcommand{\RootPath}{.}

\newcommand{\ExternalFiguresPath}{\RootPath/figures}

\newcommand{\eop}{\hspace*{\fill}~$\square$} 

\theoremstyle{plain}
\newtheorem{theorem}{Theorem}[section]
\newtheorem{proposition}[theorem]{Proposition}
\newtheorem{lemma}[theorem]{Lemma}
\newtheorem{corollary}[theorem]{Corollary}
\newtheorem{Conjecture}[theorem]{Conjecture}

\newtheorem*{theorem*}{Theorem}
\newtheorem*{proposition*}{Proposition}
\newtheorem*{lemma*}{Lemma}
\newtheorem*{corollary*}{Corollary}
\newtheorem*{Conjecture*}{Conjecture}
\newtheorem*{unverified*}{Unverified Statement}
\newenvironment{conjecture*}{\begin{Conjecture*}}{\eop\end{Conjecture*}}

\theoremstyle{definition}
\newtheorem{Definition}[theorem]{Definition}

\newtheorem{Remark}[theorem]{Remark}

\newtheorem{Remarks}[theorem]{Remarks} 
\newtheorem{Example}[theorem]{Example}

\newtheorem{Examples}[theorem]{Examples}

\newtheorem{Hypothesis}[theorem]{Hypothesis}

\newtheorem{Problem}[theorem]{Problem}

\newtheorem{Problems}[theorem]{Problems}

\newtheorem{Exercise}[theorem]{Exercise}

\newtheorem*{Definition*}{Definition}
\newenvironment{definition*}{\begin{Definition*}}{\eop\end{Definition*}}
\newtheorem*{Remark*}{Remark}
\newenvironment{remark*}{\begin{Remark*}}{\eop\end{Remark*}}
\newtheorem*{Remarks*}{Remarks} 
\newenvironment{remarks*}{\begin{Remarks*} {\ } \hspace{-.5cm} \begin{enumerate}}{\end{enumerate}\eop\end{Remarks*}} 
\newtheorem*{Example*}{Example}
\newenvironment{example*}{\begin{Example*}}{\eop\end{Example*}}
\newtheorem*{Examples*}{Examples}
\newenvironment{examples*}{\begin{Examples*} \begin{enumerate}}{\end{enumerate}\eop\end{Examples*}}
\newtheorem*{Hypothesis*}{Hypothesis}
\newenvironment{hypothesis*}{\begin{Hypothesis*}}{\eop\end{Hypothesis*}}
\newtheorem*{Problem*}{Problem}
\newenvironment{problem*}{\begin{Problem*}}{\eop\end{Problem*}}
\newtheorem*{Problems*}{Problems}
\newenvironment{problems*}{\begin{Problems*} {\ } \hspace{-.5cm} \begin{enumerate}}{\end{enumerate}\eop\end{Problems*}}
\newtheorem*{Exercise*}{Exercise}
\newenvironment{exercise*}{\begin{Exercise*}}{\eop\end{Exercise*}}

\newcommand{\mycaption}[1]{\centering{\vspace{\medskipamount}\refstepcounter{figure}\textbf{Figure~\thefigure.} {#1}}}

\newcommand{\natur}{\ensuremath{\mathbb{N}}}
\newcommand{\real}{\ensuremath{\mathbb{R}}}

\newcommand{\setcond}[2]{\left\{ #1 : #2 \right\}} 

\newcommand{\bigsetcond}[2]{\bigl\{ #1 \,  :\, #2 \bigr\}}

\newenvironment{FigTab}[2]{
	\begin{figure}[htb]
	\setlength{\unitlength}{#2}
	\begin{center}
	\begin{tabular}{#1}
}{
    \end{tabular}
    \end{center}
    \end{figure}
}

\newcommand{\IncludeGraph}[2]{
	\includegraphics[#1]{\ExternalFiguresPath/{#2}}
}

\newcommand{\intr}{\mathop{\mathrm{int}}\nolimits}

\newcommand{\bd}{\mathop{\mathrm{bd}}\nolimits}

\newcommand{\sign}{\mathop{\mathrm{sign}}\nolimits}

\newboolean{hideinwork}
\newboolean{hidefigures}
\newboolean{hideacomments}
\setboolean{hideacomments}{false} 
\setboolean{hideinwork}{false} \setboolean{hidefigures}{false}

\newcommand{\inwork}[1]{ \ifthenelse{\boolean{hideinwork}}{}{ {\bf inwork( }  #1 {\bf )} } }
\newcommand{\putfigure}[1]{ \ifthenelse{\boolean{hidefigures}}{{ \sf FIGURE }}{ #1 } }

\newcommand{\acomment}[1]{\ifthenelse{\boolean{hideacomments}}{}{ {\footnote{#1}} }} 

\setboolean{hideacomments}{true}

\newcommand{\MxL}{\left[} 
\newcommand{\MxR}{\right]} 

\newcommand{\ThmTitle}[2][]{\ifthenelse{\equal{#1}{}}{\emph{(#2).}}{\emph{(#2; #1).}}}
\newcommand{\notion}[2][]{\emph{#2}\xspace} 

\newcommand{\mforall}{\ \mbox{for all} \ }
\newcommand{\mand}{\ \mbox{and} \ }
\newcommand{\mor}{\ \mbox{or} \ }

\newcommand{\eps}{\varepsilon}

\newcommand{\ideal}{\mathop{\mathcal{I}}\nolimits}

\newcommand{\mycite}[2]{\ifthenelse{\equal{#2}{}}{\cite{#1}}{\cite[#2]{#1}}\xspace}
\newcommand{\GroetschelHenk}[1][]{\mycite{MR1976602}{#1}}
\newcommand{\BosseGroetschelHenk}[1][]{\mycite{MR2166533}{#1}}
\newcommand{\Bernig}[1][]{\mycite{Bernig98}{#1}}
\newcommand{\Ziegler}[1][]{\mycite{MR1311028}{#1}}

\newcommand{\RAGbook}[1][]{\mycite{MR1659509}{#1}}

\newcommand{\CoxLittle}[1][]{\mycite{MR1189133}{#1}}
\newcommand{\BasuPollackRoy}[1][]{\mycite{MR2248869}{#1}}

\begin{document}
\title{Notes on Algebra and Geometry of\\ Polynomial Representations}
\date{\small \today}
\author{\small Gennadiy Averkov\footnote{Work supported by the German Research Foundation within the Research Unit 468 ``Methods from Discrete Mathematics for the Synthesis and Control of Chemical Processes''}}
\maketitle

\begin{abstract}
	Consider a semi-algebraic set $A$ in $\real^d$ constructed from the sets which are determined by inequalities $p_i(x)>0$, $p_i(x) \ge 0$, or $p_i(x) = 0$ for a given list of polynomials $p_1,\ldots,p_m.$  We prove several statements that fit into the following template. Assume that in a neighborhood of a boundary point  the semi-algebraic set $A$ can be described by an irreducible polynomial $f$. Then $f$ is a factor of a certain multiplicity of some of the polynomials $p_1,\ldots,p_m.$  Special cases when $A$ is elementary closed, elementary open, a polygon, or a polytope are considered separately. 
\end{abstract}

\newtheoremstyle{itdot}{}{}{\mdseries\rmfamily}{}{\itshape}{.}{ }{}
\theoremstyle{itdot}
\newtheorem*{msc*}{2000 Mathematics Subject Classification} 

\begin{msc*}
  Primary: 14P10, Secondary: 52B05
\end{msc*}


\newtheorem*{keywords*}{Key words and phrases}

\begin{keywords*}
Irreducible polynomial, polygon, polytope, polynomial representation, real algebraic geometry, semi-algebraic set
\end{keywords*}

\section{Introduction} 

In what follows $x:=(x_1,\ldots,x_d)$ is a variable vector  in $\real^d$ ($d \in \natur$). The origin in $\real^d$ is denoted by $o.$ Given $c \in \real^d$ and $\rho>0$ by $B^d(c,\rho)$ we denote the open Euclidean ball with center $c$ and radius $\rho.$ The abbreviations $\intr$ and $\bd$ stands for the interior and boundary, respectively. By $\dim$ we denote the dimension.  As usual, $\real[x]:=\real[x_1,\ldots,x_d]$ denotes the ring of polynomials in variables $x_1,\ldots,x_d$ and coefficients in $\real.$ A set $A \subseteq \real^d$ which can be represented by 
$$
	A=\bigcup_{i=1}^k \setcond{x \in \real^d}{f_{i,1}(x) > 0, \ldots, f_{i,s_i}(x)>0, \  g_i(x)=0},
$$
where $i \in \{1,\ldots,k\}, \ j \in \{1,\ldots,s_i\}$ and $f_{i,j}, \ g_i \in \real[x]$,  is called \notion[semi-algebraic set]{semi-algebraic}.  Information on semi-algebraic sets can be found in \cite{MR1393194}, \RAGbook, and \cite{MR2248869}. Obviously, every semi-algebraic set $A$ can be expressed by 
\begin{equation} \label{A:rep}
	A= \setcond{x \in \real^d}{ \Phi\bigl((\sign p_1(x) \in E_1),\ldots,(\sign p_m(x) \in E_m)\bigr)},
\end{equation}  where $\Phi$ is a boolean formula, $p_1,\ldots,p_m \in \real[x]$, and $E_1,\ldots,E_m$ are non-empty subsets of $\{0,1\}$. Vice versa, every set $A$ given by \eqref{A:rep} is semi-algebraic;
see \RAGbook[Proposition~2.2.4] and \BasuPollackRoy[Corollary~2.75]. We call \eqref{A:rep} a \notion[representation of a semi-algebraic set]{representation} of $A$ by polynomials $p_1,\ldots,p_m.$ We distinguish several particular types of semi-algebraic sets. Let us introduce the following notations: 
\begin{eqnarray}
 	(p_1,\ldots,p_m)_{\ge 0} & := & \setcond{x \in \real^d}{p_1(x) \ge 0,\ldots, p_m(x) \ge 0}, \label{s:closed:def} \\
	(p_1,\ldots,p_s)_{> 0} &:=& \setcond{x \in \real^d}{p_1(x) > 0,\ldots, p_m(x) > 0}, \label{s:open:def} \\
	Z(p_1,\ldots,p_m) &:=& \setcond{x \in \real^d}{p_1(x)=0, \ldots, p_m(x)=0}. \label{alg-set:def}
\end{eqnarray}
Sets representable by \eqref{s:closed:def}, \eqref{s:open:def}, and \eqref{alg-set:def}, respectively, are called \notion{elementary closed semi-algebraic}, \notion{elementary open semi-algebraic}, and \notion{algebraic}, respectively. 

Now we are ready to formulate our main results.  First we give an informal interpretation of Theorem~\ref{m:thm}. Let $f$ be an irreducible polynomial such that $Z(f)$ is a $(d-1)$-dimensional algebraic surface. Consider a semi-algebraic set $A$ given by \eqref{A:rep}. If the boundary of $A$ coincides locally with a part of $Z(f),$ then $f$ is a factor of some $p_i.$ If $A$ coincides locally with a part of $(f)_{\ge 0},$ then $f$ is an odd-multiplicity factor of some $p_i.$ Furthermore, if in a neighborhood of a boundary point the set $A$ coincides locally with a part of $Z(f)$, then $f$ is a factor of at least two different polynomials $p_i$ or an even-multiplicity factor of at least one polynomial $p_i.$

\begin{theorem} \label{m:thm} Let $A$ be a semi-algebraic set in $\real^d$ given by \eqref{A:rep} and let $f$ be a polynomial irreducible over $\real[x]$. Then the following statements hold true. 
\begin{enumerate}[I.]
	\item \label{08.02.04,09:38A} One has 
	$\bd A \subseteq \bigcup_{i=1}^m Z(p_i)$.
	\item \label{08.02.09,20:41A} If 
		\begin{equation} \label{08.02.13,15:03a}
			\dim (\bd A \cap Z(f) )= d-1,
		\end{equation} 
		then $f$ is a factor of $p_i$ for some $i \in \{1,\ldots,m\}$.
	\item \label{08.02.09,20:46A} If there exist $a \in Z(f)$ and $\eps>0$ such that 
		\begin{eqnarray}
			\dim (Z(f) \cap B^d(a,\eps)) & = & d-1, \label{08.02.11,10:04A} \\
			(f)_{\ge 0} \cap B^d(a,\eps) & = &  A \cap B^d(a,\eps), \label{08.02.11,10:05A}
		\end{eqnarray} 
		 then  \eqref{08.02.13,15:03a} is fulfilled and, moreover, $f$ is an odd-multiplicity factor of $p_i$ for some $i \in \{1,\ldots,m\}$. 
		\item \label{08.02.11,00:32A} If there exist $a \in Z(f)$ and $\eps>0$ such that 
		\begin{eqnarray*}
			\dim (Z(f) \cap B^d(a,\eps)) &= & d-1, \\ 
			(f)_{> 0} \cap B^d(a,\eps) & = & A \cap B^d(a,\eps), 
		\end{eqnarray*} 
		then \eqref{08.02.13,15:03a} is fulfilled and, moreover, $f$ is an odd-multiplicity factor of $p_i$ for some $i \in \{1,\ldots,m\}$.
\end{enumerate} 
\end{theorem} 

We remark that \eqref{08.02.11,10:04A} cannot be replaced by the weaker condition $\dim Z(f) = d-1$ and $Z(f) \cap B^d(a,\eps) \ne \emptyset,$ since the  algebraic set $Z(f)$ corresponding to an irreducible polynomial $f$ can have ``parts'' of dimensions strictly smaller than $\dim Z(f).$  In fact, for $d=2$ the irreducible polynomial $f(x):=x_1^2+x_2^2-x_1^3$ generates the cubic curve $Z(f)$ with isolated point at the origin. For $d=3,$ for the irreducible polynomial $f(x)=x_3^2 \, x_1 - x_2^2$ the set $Z(f)$ is the well-known \notion{Whitney umbrella},  which is a two-dimensional algebraic surface with the one-dimensional ``handle'' $Z(x_2,x_3).$

The rest of the introduction is devoted to statements for some special semi-algebraic sets (and special representations of semi-algebraic sets).

\begin{corollary} \label{cor1}
	Let $A$ be a semi-algebraic set given by $$A= \setcond{x \in \real^d}{ \Phi\bigl((p_1(x) \ge 0),\ldots,(p_m(x) \ge 0)\bigr)},$$ where $\Phi$ is a boolean formula and $p_1,\ldots,p_m \in \real[x] \setminus \{0\}$, and let $f$ be a polynomial irreducible over $\real[x].$ Then the following statements hold true.
\begin{enumerate}[I.]
	\item \label{08.02.09,20:47A} If there exist $b \in Z(f)$ and $\eps>0$ such that 
		\begin{eqnarray} 
			\dim(Z(f) \cap B^d(b,\eps)) &=& d-1, \label{08.02.11,10:06A} \\ 
			Z(f) \cap B^d(b,\eps) &=& A \cap B^d(b,\eps), \label{08.02.11,10:07A}
		\end{eqnarray} 
		 then \eqref{08.02.13,15:03a} is fulfilled, and furthermore $f$ is a factor of $p_i$ and $p_j$ for some $i,  \, j \in \{1,\ldots, m\}$ with $i\ne j$  or $f$ is an even-multiplicity factor of $p_i$ for some $i \in \{1,\ldots,m\}.$
	\item \label{08.04.13,14:27A} If there exist $a, \,  b \in \real^d$ and $\eps>0$ such that equalities \eqref{08.02.11,10:04A},  \eqref{08.02.11,10:05A}, \eqref{08.02.11,10:06A}, and \eqref{08.02.11,10:07A} are fulfilled, then $f$ is a factor of $p_i$ and an odd-multiplicity factor of $p_j$ for some $i,  \, j \in \{1,\ldots,m \}$ with $i \ne j$. 
\end{enumerate} 
\eop
\end{corollary}

\begin{corollary} \label{m:cor:closed} 
	Let $p_1,\ldots,p_m \in \real[x] \setminus \{0\}$ and $A:=(p_1,\ldots,p_m)_{\ge 0}.$  Let $f$ be a polynomial irreducible over $\real[x].$ Assume that there exist $b \in Z(f)$ and $\eps>0$ such that equalities \eqref{08.02.11,10:06A} and \eqref{08.02.11,10:07A} are fulfilled and additionally 
		\begin{equation} \label{08.04.11,12:34}
			\dim (\intr A \cap Z(f)) =d-1.
		\end{equation} 
		 Then $f$ is a factor of $p_i$ for some $i \in \{1,\ldots,m \}$ and, for every $i \in \{1,\ldots,m\}$ such that $p_i$ is divisible by $f$, the factor $f$ of $p_i$ has even multiplicity. 
	\eop
\end{corollary} 

\begin{corollary} \label{m:cor:open} 
	Let $p_1,\ldots,p_m \in \real[x] \setminus \{0\}$ and $A:=(p_1,\ldots,p_m)_{> 0}.$  Let $f$ be a polynomial irreducible over $\real[x].$  Assume that there exist $b \in Z(f)$ and $\eps>0$ such that 
		\begin{eqnarray}
			\dim (\bd A \cap Z(f) \cap B^d(b,\eps)) &= & d-1, \label{08.02.11,16:45} \\
			B^d(b,\eps) \setminus Z(f) & = & A \cap B^d(b,\eps). \label{08.02.11,16:46}
		\end{eqnarray} 
		Then $f$ is a factor of $p_i$ for some $i \in \{1,\ldots,m\}$ and, for every $i \in \{1,\ldots,m\}$ such that $p_i$ is  divisible by $f$, the factor $f$ of $p_i$ has even multiplicity. 
	\eop
\end{corollary}

A subset $P$ of $\real^d$ is said to be a \notion{polytope} if $P$ is the convex hull of a non-empty and finite set of points; see \Ziegler. It is known that a set $P$ in $\real^d$ is a polytope if and only if $P$ is non-empty, bounded, and can be represented by $P=(p_1,\ldots,p_m)_{\ge 0}$, where $p_1,\ldots,p_m \in \real[x]$ ($m \in \natur$) are of degree one (the so-called \notion{H-representation}). Thus, polytopes are just special elementary closed semi-algebraic sets. The study of polynomial representations of polygons and polytopes was initiated in \Bernig and \GroetschelHenk; see also the survey  \cite{Henk06PolRep}. In \GroetschelHenk it was noticed that, if a $d$-dimensional polytope $P$ is represented by 
\begin{equation} \label{p:rep}
	P = (q_1,\ldots,q_m)_{\ge 0}
\end{equation} 
with $q_1,\ldots,q_m \in \real[x]$, then $m \ge d.$ In \BosseGroetschelHenk it was conjectured that every $d$-dimensional polytope in $\real^d$ can be represented by \eqref{p:rep} with $m=d$. This conjecture was confirmed in \cite{AveHenkRepSimplePolytopes} for simple polytopes (see also \cite{Ave:ArXiv:0804.2134} for further generalizations). We recall that a $d$-dimensional polytope is called \notion[simple polytope]{simple} if each of its vertices is contained in precisely $d$ facets. We refer to  \cite[Chapter~5]{MR1393194}  and \RAGbook[\S6.5 and \S10.4] for results on minimal representations of general elementary semi-algebraic sets. We are able to derive some necessary conditions on representations of polytopes consisting of $d$ polynomials. 

\begin{corollary} \label{polyt:cor}
	Let $P$ be a $d$-dimensional polytope in $\real^d$ with $m$ facets such that $$P=(p_1,\ldots,p_m)_{\ge 0} = (q_1,\ldots,q_d)_{\ge 0},$$ 
	where $p_1,\ldots,p_m, \, q_1,\ldots,q_d \in \real[x]$ and $p_1,\ldots,p_m$ are of degree one. Then every $p_i, \ i \in \{1,\ldots,m\},$ is a factor of precisely one polynomial $q_j$ with $j \in \{1,\ldots,d\}.$ Furthermore, for $i$ and $j$ as above, the factor $p_i$ of $p_j$ is of odd multiplicity. \eop
\end{corollary} 

Corollary~\ref{polyt:cor} improves Proposition~2.1(i) from \cite{MR1976602}. In \Bernig it was shown that every convex polygon $P$ in $\real^2$ can be represented by two polynomials. We are able to determine the precise structure of  such minimal representations.

\begin{corollary} \label{polyg:cor} Let $P$ be a convex polygon in $\real^2$ with $m \ge 7$ edges and let  
$$P  = (p_1,\ldots,p_m)_{\ge 0} = (q_1,q_2)_{\ge 0},$$
where $p_1,\ldots,p_m, \, q_1, \, q_2 \in \real[x]$ and $p_1,\ldots,p_m$ are of degree one. 
Then there exist $k_1,\ldots,k_m \in \natur$ and $g_1, \ g_2 \in \real[x]$ such that $\{q_1,q_2\}=\{ p_1^{k_1} \cdot \ldots \cdot p_m^{k_m} \, g_1, g_2\}$ and the following conditions are fulfilled:
\begin{enumerate} 
	\item \label{08.02.14,16:37} $k_1,\ldots,k_m$ are odd ; 
	\item \label{08.02.14,16:38} $g_1, g_2$ are not divisible by $p_i$ for every $i \in \{1,\ldots,m\}$;
	\item \label{08.02.14,16:39} $g_2(y)=0$ for every vertex $y$ of $P.$ 
\end{enumerate} 
\eop
\end{corollary} 

It is not hard to see  that the  the set $(q_1,q_2)_{\ge 0}$ in Corollary~\ref{polyg:cor} does not depend on the concrete choice of odd numbers $k_1,\ldots,k_m$. More precisely, for $g_1, \ g_2$ as in Corollary~\ref{polyg:cor} we have $P=(p_1 \cdot \ldots \cdot p_m \, g_1, g_2)_{\ge 0}.$   In \Bernig the polynomials $q_1, \, q_2$ representing $P$ were defined in such a way  that $g_1=1$ and $k_1=\cdots = k_m=1$; see Fig.~\ref{simp:polyt:2d:fig} for an illustration of this result and Corollary~\ref{polyg:cor}. We also remark that the assumption  $m \ge 7$ cannot be relaxed in general, since Corollary~\ref{polyg:cor} would not hold if $P$ is a centrally symmetric hexagon. In fact, assume that $P$ is a centrally symmetric hexagon and $p_1,\ldots,p_6$ are polynomials of degree one such that $Z(p_1) \cap P,\ldots, Z(p_6) \cap P$ are consecutive edges of $P.$ Then $P=(q_1,q_2)_{\ge 0}$ for $q_1:=p_1 \, p_3 \, p_5$ and $q_2:= p_2 \, p_4 \, p_6$; see Fig.~\ref{sym:hex}. It will be seen from the proof of Corollary~\ref{polyg:cor} that the assumption $m \ge 7$ can be relaxed to $m \ge 5$ for the case when $P$ does not have parallel edges. 

	\begin{FigTab}{c}{0.6mm}
	\begin{picture}(250,50)
	\put(0,2){\IncludeGraph{width=50\unitlength}{BernigSetP0.eps}}
	\put(20,25){\scriptsize $(q_1)_{\ge 0}$}
	\put(80,2){\IncludeGraph{width=50\unitlength}{BernigSetP1.eps}}
	\put(97,25){\scriptsize $(q_2)_{\ge 0}$}
	\put(155,25){\IncludeGraph{width=12\unitlength}{RightArrow.eps}}
	\put(190,2){\IncludeGraph{width=50\unitlength}{BernigP.eps}}
	\put(210,25){\scriptsize $P$}
	\end{picture} 
	\\
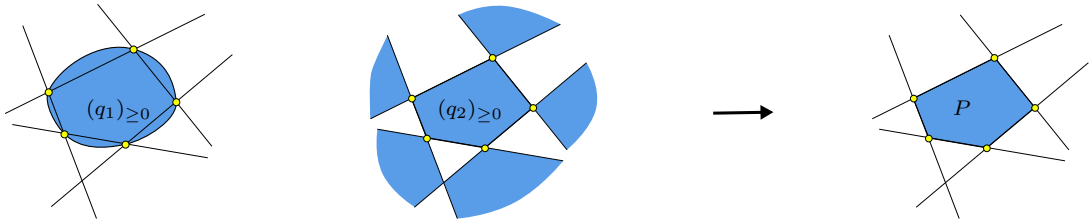
	\parbox[t]{0.98\textwidth}{\mycaption{Illustration to Corollary~\ref{polyg:cor} and the result on representation of convex polygons by two polynomials\label{simp:polyt:2d:fig}}} 
	\end{FigTab}

	\begin{FigTab}{c}{0.55mm}
	\begin{picture}(270,77)
	\put(0,10){\IncludeGraph{width=63\unitlength}{hex-p1p3p5.eps}}
	\put(28,39){\scriptsize $(q_1)_{\ge 0}$}
	\put(110,0){\IncludeGraph{width=63\unitlength}{hex-p2p4p6.eps}}
	\put(133,39){\scriptsize $(q_2)_{\ge 0}$}
	\put(190,37){\IncludeGraph{width=12\unitlength}{RightArrow.eps}}
	\put(225,22){\IncludeGraph{width=35\unitlength}{sym-hex.eps}}
	\put(240,35){$P$}
	\end{picture}  \\
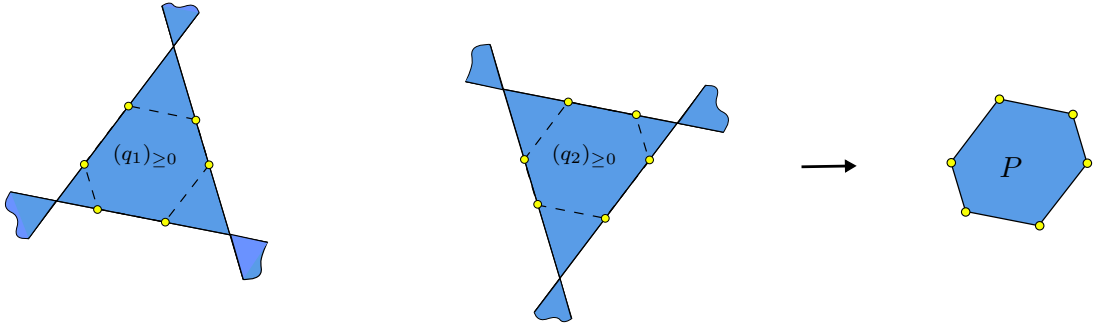
	\parbox[t]{0.99\textwidth}{\mycaption{Centrally symmetric hexagon $P$ represented by $P=(q_1,q_2)_{\ge 0}$ for $q_1=p_1 \, p_3 \, p_4$ and $q_2=p_2 \, p_4 \, p_6$\label{sym:hex}}} 
	\end{FigTab}

\section{Examples} \label{examples:sect}

We wish to give several examples illustrating the presented results. Each of the examples below is supplied with a figure referring to the case $d=2.$ Let 
$$
	A : = \bigsetcond{ x \in \real^d}{ x_d >0 \mand \left( (x_1-1)^2 + x_2^2 + \cdots + x_d^2 \le 1 \mor x_1^2 + x_2^2 + \cdots + x_d^2 \le 1 \right)},
$$
see Fig.~\ref{GeneralSemialg}. By Theorem~\ref{m:thm}, if $A$ is given by \eqref{A:rep}, then the polynomials $x_d,$ $(x_1-1)^2 + x_2^2 + \cdots + x_d^2 - 1$, $(x_1+1)^2 + x_2^2 + \cdots + x_d^2 - 1$ are factors of odd multiplicity of some of the polynomials $p_1,\ldots,p_m.$ 

	\begin{FigTab}{c}{0.35mm}
	\begin{picture}(100,30)
	\put(0,5){\IncludeGraph{width=95\unitlength}{GeneralSemialg.eps}}
  	\put(40,15){\large $A$}
	\end{picture} 
	\\
 	\parbox[t]{0.45\textwidth}{\mycaption{Illustration to Theorem~\ref{m:thm}\label{GeneralSemialg}}} 
	\end{FigTab}

 The set 
\begin{align}
	A & := \bigsetcond{x \in \real^d}{ (1- x_1^2 - \cdots - x_d^2) \, (x_d+2)^2 \ge 0}, \label{08.04.18,13:01} \\
	    & = \bigsetcond{x \in \real^d}{ (1-x_1^2 - \cdots - x_d^2) \, (x_d+2) \ge 0, \ x_d+2 \ge 0 }, \label{08.04.18,13:02}
\end{align}
depicted in Fig.~\ref{RisingSun} is the disjoint unit of a closed unit ball centered at $o$ and a hyperplane given by the equation $x_d+2=0.$ By Corollary~\ref{cor1}(\ref{08.02.09,20:47A}), if $A$ is given by \eqref{A:rep}, then $x_d+2$ is a factor  of at least two polynomials $p_i$ or a factor of even multiplicity of at least one polynomial $p_1,\ldots,p_m$. From \eqref{08.04.18,13:01} and \eqref{08.04.18,13:02} we see that both of these possibilities are indeed realizable.  Fig.~\ref{LocalStructExample} depicts the semi-algebraic set 
\begin{align}
	A & := \setcond{x \in \real^2}{x_d \ge 0, \  (1-x_1^2- \cdots - x_d^2) \, x_d \ge 0},  \nonumber \\
		& = \setcond{x \in \real^2}{x_d \ge 0, \ (1-x_1^2- \cdots - x_d^2) \, x_d^2 \ge 0 }, \label{08.04.18,13:04}
\end{align}
By  Corollary~\ref{cor1}(\ref{08.04.13,14:27A}), if $A$ is given by \eqref{A:rep} with $E_1=\ldots=E_m=\{0,1\},$ the polynomial $x_d$ is a factor of at least two polynomials $p_i$ and an odd-multiplicity factor of at least one polynomial $p_i.$ By \eqref{08.04.18,13:04} we see that the above conclusion cannot be strengthened. In fact, \eqref{08.04.18,13:04} provides a representation $A=(p_1,p_2)_{\ge 0}$ such that $x_d$ is an odd-multiplicity factor of precisely one polynomial $p_i$.

	\begin{FigTab}{cc}{0.5mm}
	\begin{picture}(100,80)
	\put(0,5){\IncludeGraph{width=95\unitlength}{RisingSun.eps}}
 	\put(45,50){\large $A$}
	\put(5,7){\large $A$}
	\end{picture}  
	&
	\begin{picture}(100,80)
	\put(-3,30){\IncludeGraph{width=105\unitlength}{LocalStructExample.eps}}
 	\put(40,40){\large $A$}
 	\put(0,34){\large $A$}
 	\put(37,24){\scriptsize $B^d(a,\eps)$}
 	\put(80,24){\scriptsize $B^d(b,\eps)$} 
	\end{picture}  
	\\
	\parbox[t]{0.45\textwidth}{\mycaption{Illustration to Corollary~\ref{cor1}(\ref{08.02.09,20:47A})\label{RisingSun}}} 
	& \parbox[t]{0.45\textwidth}{\mycaption{Illustration to Corollary~\ref{cor1}(\ref{08.04.13,14:27A})\label{LocalStructExample}}} 
	\end{FigTab}

Fig.~\ref{Saturn} presents the semi-algebraic set 
$$	
	A: = \setcond{x \in \real^d}{ (1- x_1^2 - \cdots - x_d^2) \, x_d^2 \ge 0}.
$$
which serves as an illustration of Corollary~\ref{m:cor:closed}. By Corollary~\ref{m:cor:closed}, if $A=(p_1,\ldots,p_m)_{\ge 0}$ for  $p_1,\ldots,p_m \in \real[x]$, some of these polynomials are divisible by $x_d,$ and furthermore, if $p_i$ is divisible by $x_d$, the multiplicity of the factor $x_d$ of $p_i$ is even. Fig.~\ref{LocalStructExampleOpen} depicts the semi-algebraic set $$A:= \setcond{x \in \real^2}{ (1-x_1^2- \cdots - x_d^2) \, x_d^2 >0 }$$ 
illustrating Corollary~\ref{m:cor:open}. By Corollary~\ref{m:cor:open}, if $A=(p_1,\ldots,p_m)_{>0}$ for some polynomials $p_1,\ldots, p_m \in \real[x],$ then $x_d$ is a factor of at least one $p_i$ and $f$ cannot be a factor of $p_i$ of odd multiplicity. We notice that Corollary~\ref{m:cor:open} is in a certain sense an analogue of Corollary~\ref{m:cor:closed} for elementary open semi-algebraic sets (since the conclusions of both the corollaries are the same). 

	\begin{FigTab}{cc}{0.4mm}
	\begin{picture}(100,60)
	\put(0,5){\IncludeGraph{width=95\unitlength}{Saturn.eps}}
 	\put(49,34){\large $A$}
 	\put(3,31){\large $A$}
	\end{picture}  
	& 
	\begin{picture}(75,60)
	\put(0,5){\IncludeGraph{width=70\unitlength}{LocalStructExampleOpen.eps}}
 	\put(30,13){\large $A$}
 	\put(33,33){\scriptsize $B^d(b,\eps)$} 
	\end{picture}
	\\
	\parbox[t]{0.45\textwidth}{\mycaption{Illustration to Corollary~\ref{m:cor:closed}\label{Saturn}}} 
	& \parbox[t]{0.45\textwidth}{\mycaption{Illustration to Corollary~\ref{m:cor:open}\label{LocalStructExampleOpen}}} 
	\end{FigTab}

Finally, we present examples of semi-algebraic sets for which we can verify that they are not elementary semi-algebraic (see also similar examples given in \cite[p.~24]{MR1393194}). We define the closed semi-algebraic set
\begin{equation*}
	A:= \bigsetcond{x \in \real^d}{x_d = 0 \mor (x_1-3)^2+x_2^2 + \cdots x_d^2 \le 1 \mor \left( x_1^2+x_2^2 + \cdots x_d^2 \le 1 \mand x_d \ge 0 \right) },
\end{equation*}
see Fig.~\ref{NonElementaryExample}. We can show that $A$ is not elementary closed. In fact, let us assume the contrary, that is $A=(p_1,\ldots,p_m)_{\ge 0}$ for some polynomials $p_1,\ldots,p_m \in \real[x].$ Then, by Theorem~\ref{m:thm}(\ref{08.02.09,20:46A}) applied for $a=o$ and $0<\eps<1,$ we get that $x_d$ is a factor of odd multiplicity of $p_i$ for some $ i \in \{1,\ldots,m\}.$ Since  \eqref{08.04.11,12:34} is fulfilled for $f=x_d$, we can apply Corollary~\ref{m:cor:closed} obtaining that $x_d$ is a factor of even multiplicity of $p_i,$ a contradiction. Now we introduce the open semi-algebraic set 
\begin{align*}
	 A  := \bigsetcond{x \in \real^d}{x_1^2+x_2^2 + \cdots x_d^2 <1 \mand x_d >0 \mor  (x_1-3)^2+x_2^2 + \cdots x_d^2 < 1  \mand  x_d \ne 0 },
\end{align*}
see Fig.~\ref{NonElementaryExampleOpen}. By Theorem~\ref{m:thm}(\ref{08.02.11,00:32A}) and Corollary~\ref{m:cor:open} (applied for $f(x)=x_d$) $A$ is not elementary open.

	\begin{FigTab}{cc}{0.5mm}
	\begin{picture}(110,40)
	\put(0,5){\IncludeGraph{width=105\unitlength}{NonElementaryExample.eps}}
 	\put(22,25){\large $A$}
 	\put(80,25){\large $A$}
 	\put(0,22){$A$}
	\put(22,12){\scriptsize $B^d(a,\eps)$}
	\end{picture}  
	&
	\begin{picture}(110,40)
	\put(0,5){\IncludeGraph{width=105\unitlength}{NonElementaryExampleOpen.eps}}
 	\put(22,25){\large $A$}
 	\put(80,25){\large $A$}
	\put(22,12){\scriptsize $B^d(a,\eps)$}
	\end{picture}  
	\\
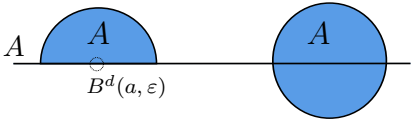
	\parbox[t]{0.45\textwidth}{\mycaption{A closed semi-algebraic set which is not elementary closed\label{NonElementaryExample}}} 
	&
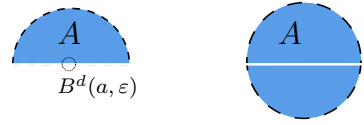
	\parbox[t]{0.45\textwidth}{\mycaption{An open semi-algebraic set which is not elementary open\label{NonElementaryExampleOpen}}} 
	\end{FigTab}

\section{Preliminaries and auxiliary statements} \label{prelim}

A polynomial $f \in \real[x]$ is said to be \notion[irreducible polynomial]{irreducible} over $\real[x]$ if $f$ is non-constant and $f$ cannot be represented as a product of two non-constant polynomials over $\real[x].$ A polynomial $p$ is said to be a \notion{factor} of $q$ if $q=p g$ for some polynomial $g.$ An irreducible factor $f$ of $p$ is said to have \notion{multiplicity} $k \in \natur$ if $f^k$ is a factor of $p$ but $f^{k+1}$ is not a factor of $p.$  

Below we give background information on commutative algebra and algebraic geometry; see also \cite{MR2297716} and \CoxLittle. Let $R$ be a commutative  ring. Then a subset $I$ of $R$ is said to be an \emph{ideal} if $I$ is an additive group and for every $f \in I$ and $g \in R$ one has $f g \in I.$ An ideal $I$ of $R$ is said to be \emph{prime} if for every product $f g \in I$  with $f, g \in R,$ one has $f \in I$ or $g \in I.$  The \notion[dimension]{(Krull) dimension} of a commutative ring $R$ is the maximal length of a sequence of prime ideals $I_1,\ldots, I_k$ satisfying
 $I_1 \varsubsetneq I_1 \varsubsetneq \cdots \varsubsetneq I_k \varsubsetneq R.$ The \notion{factor ring} $R/I$ is defined as the set $\setcond{x + I}{x \in R}$ with the addition and multiplication induced by $R.$ 

\begin{lemma} \label{PropDimRings}
	Let $R$ be a commutative ring and $I, J$ be ideals in $R$ such that $I$ is prime and $I \subseteq J.$ Then $\dim (R/J) \le \dim (R/I)$ with equality if and only if $I=J.$ 
\end{lemma}
\begin{proof}
	It is known that every ideal $X$ of $R/J$ has the form $X=P/J:=\setcond{x+I}{x \in P},$ where $P$ is an ideal in $R$ with $J \subseteq P$; see \cite[p.~9 of Chapter~2]{MR2297716}. Furthermore, $X$ is prime in $R/J$ if and only if $P$ is prime in $R.$ Using this observation we readily get that the dimension of $R/J$ is the maximal length of sequence of prime ideals $I_1,\ldots, I_k$ satisfying
 $J \subseteq I_1 \varsubsetneq I_1 \varsubsetneq \cdots \varsubsetneq I_k \varsubsetneq R.$ If $I$ is properly contained in $J,$ then $I, I_1,\ldots,I_k$ is the chain of prime ideals containing $I,$ and we get that the dimension of $R/I$ is strictly larger than the dimension of $R/J.$ 
\end{proof}


 An algebraic set $V \subseteq \real^d$ is said to be \notion[irreducible algebraic set]{irreducible} if whenever $V$ is represented by $V=V_1 \cup V_2,$ where $V_1, V_2 \subseteq \real^d$ are algebraic sets, it follows that $V_1=V$ or $V_2=V.$ Given a set $A \subseteq \real^d,$ we introduce the ideal 
$$\ideal(A):=\setcond{p \in \real[x]}{p(x) = 0 \ \mforall \, x \in A}.$$ 
It is known that an algebraic set $V$ in $\real^d$ is irreducible if and only if the ideal $\ideal(V)$ of the ring $\real[x]$ is prime; see \CoxLittle[Proposition~3, p.~195] and \RAGbook[Theorem~2.8.3(ii)]. The notion of dimension of a (semi-algebraic) set can be defined in several equivalent ways; for details see \RAGbook[\S2.8].  We shall employ the following algebraic definition. The \notion{dimension} of a semi-algebraic set $A \subseteq \real^d$ is defined as the dimension of the ring $\real[x]/\ideal(A)$. Let $A_1, \ldots, A_m$ be semi-algebraic sets in $\real^d.$ Then
\begin{equation} \label{dim:union}
	\dim (A_1 \cup \ldots \cup A_m) = \max \setcond{\dim A_i}{1 \le i \le m},
\end{equation}
see \cite[Proposition~2.8.5(i)]{MR1659509}.

We present several statements devoted to irreducible polynomials over $\real[x]$ that define $(d-1)$-dimensional algebraic sets. Given a polynomial $p \in \real[x]$, by $\nabla p$ we denote the gradient of $p.$  The following statement can be found in \RAGbook[Theorem~4.5.1].
\begin{theorem} \label{ReIdealThm}
	Let $f$ be a polynomial irreducible over $\real[x]$. Then the following conditions are equivalent.
	\begin{enumerate}[(i)]
		\item \label{08.01.31,14:06} 
		$\ideal(Z(f))=\setcond{ f g  }{g \in \real[x]}.$
		\item \label{08.01.31,14:07} The polynomial $f$ has a \notion{non-singular zero}, i.e., for some $y \in \real^d$ one has $f(y)=0$ and $\nabla f(y) \ne o.$ 
		\item \label{08.01.31,14:09} $\dim Z(f) = d-1.$ 
	\end{enumerate} 
	\eop
\end{theorem}

\begin{lemma} \label{08.01.30,17:53} Let $f, \, p \in \real[x].$ Let $f$ be irreducible over $\real[x]$ and let $\dim Z(f)=d-1.$ Then the following conditions are equivalent.
\begin{enumerate}[(i)]
	\item \label{08.01.31,14:15} $\dim(Z(f) \cap Z(p))=d-1$. 
	\item \label{08.01.31,14:18} $Z(f) \subseteq Z(p).$ 
	\item \label{08.01.31,14:17} $f$ is a factor of $p.$ 
\end{enumerate}
\end{lemma}
\begin{proof} 	It suffices to verify \eqref{08.01.31,14:15} $\Rightarrow$ \eqref{08.01.31,14:18} and \eqref{08.01.31,14:18} $\Rightarrow$ \eqref{08.01.31,14:17}, since the implications \eqref{08.01.31,14:17} $\Rightarrow$ \eqref{08.01.31,14:18} $\Rightarrow$ \eqref{08.01.31,14:15} are trivial.

	\emph{(\ref{08.01.31,14:15}) $\Rightarrow$ (\ref{08.01.31,14:18}):} 	Assume that \eqref{08.01.31,14:15} is fulfilled. 
      The implication \eqref{08.01.31,14:09} $\Rightarrow$ \eqref{08.01.31,14:06} of Theorem~\ref{ReIdealThm} yields 
	$\ideal(Z(f))= \setcond{ f g}{g \in \real[x]}.$
	Consequently, since $f$ is irreducible, the ideal $\ideal(Z(f))$ is prime.
	Obviously, $\ideal(Z(f)) \subseteq \ideal(Z(f,p)).$ Furthermore,
	$$
		\dim \real[x] / \ideal(Z(f,p)) = \dim Z(f,p) \stackrel{\eqref{08.01.31,14:15}}{=} d-1  = \dim Z(f)  = \dim  \real[x] / \ideal(Z(f))
	$$
	and,  by Lemma~\ref{PropDimRings}, it follows that $\ideal(Z(f))=\ideal(Z(f,p)).$ The latter equality yields $Z(f)=Z(f,p)$; see \CoxLittle[Proposition~8, p.~34]. Since $Z(f,p)=Z(f)\cap Z(p),$ the statement \eqref{08.01.31,14:18} readily follows.

	\emph{(\ref{08.01.31,14:18}) $\Rightarrow$ (\ref{08.01.31,14:17}):} Since $Z(f) \subseteq Z(p)$ it follows that $\ideal(Z(p)) \subseteq \ideal(Z(f))$ and hence $p \in \ideal(Z(f)).$ But then, by the implication \eqref{08.01.31,14:09} $\Rightarrow$ \eqref{08.01.31,14:06} of Theorem~\ref{ReIdealThm}, it follows that $f$ is a factor of $p.$ 
\end{proof}

As a direct consequence of the implication (\ref{08.01.31,14:15}) $\Rightarrow$ (\ref{08.01.31,14:17}) of Lemma~\ref{08.01.30,17:53} we obtain
\begin{lemma} \label{loc:coinc:irred}
	Let $f$ and $g$ be polynomials irreducible over $\real[x]$  and let $\dim Z(f) = \dim Z(g) = d-1.$ Then $\dim (Z(f) \cap Z(g))=d-1$ if and only if $f$ and $g$ coincide up to a constant multiple. \eop
\end{lemma} 

\begin{proposition}  \label{08.02.04,15:23} Let $f$ be a polynomial irreducible over $\real[x]$ and such that $\dim Z(f)=d-1$. Then 
$$\dim Z(f, \frac{\partial}{\partial x_1} f, \ldots, \frac{\partial}{\partial x_d} f ) \le d-2.
$$
\end{proposition} 
\begin{proof} 
 	Even though this statement is known (see \RAGbook[Proposition~3.3.14]), we wish to give a short proof. We assume that $\dim Z(f, \frac{\partial}{\partial x_1} f, \ldots, \frac{\partial}{\partial x_d} f )=d-1.$ Then, by Lemma~\ref{08.01.30,17:53}, one has $Z(f) \subseteq Z(\frac{\partial}{\partial x_1} f, \ldots, \frac{\partial}{\partial x_d} f )$, a contradiction to the implication \eqref{08.01.31,14:09} $\Rightarrow$ \eqref{08.01.31,14:07} of Theorem~\ref{ReIdealThm}.
\end{proof}

\section{The proofs} \label{proofs}

Now we are ready to prove the main result and its corollaries. In the proofs we shall deal with polynomials $p_1,\ldots,p_m.$ Throughout the proofs $f_1,\ldots, f_n$ will denote the polynomials irreducible over $\real[x]$ which are involved in the \notion{prime factorization} of the product $p_1 \cdot \ldots \cdot p_m$ (see \CoxLittle[p.~149]), i.e.
$$
	p_1 \cdot \ldots \cdot p_m = f_1^{s_1} \cdot \ldots \cdot f_n^{s_n}
$$
for some $s_1,\ldots, s_n \in \natur$ and for every $i, j \in \{1,\ldots,n\}$ with $i \ne j$ the polynomials $f_i$ and $f_j$ do not coincide up to a constant multiple.

\begin{proof}[Proof of Theorem~\ref{m:thm}] For $x \in \real^d$ we define 
$$
	\Psi(x) := \Phi\bigl((\sign p_1(x) \in E_1),\ldots,(\sign p_m(x) \in E_m)\bigr).
$$
	
\emph{Part~\ref{08.02.04,09:38A}:} Let $x_0 \not \in \bigcup_{i=1}^m Z(p_i),$ that is $p_i(x_0) \ne 0$ for every $i=1,\ldots,m.$ Then there exists an $\eps>0$ such that the sign of every $p_i(x), \ i \in \{1,\ldots,m\},$ remains constant on $B^d(x_0,\eps).$ It follows that $\Psi(x)$ is constant for $x \in B^d(x_0,\eps).$ Consequently, either $B^d(x_0,\eps) \subseteq A$ or $B^d(x_0,\eps) \cap A = \emptyset.$ Hence $x_0$ is either an interior or an exterior point of $A$, and we get the conclusion of Part~\ref{08.02.04,09:38A}.

\emph{Part~\ref{08.02.09,20:41A}:} By Part~\ref{08.02.04,09:38A} we have $\bd A \subseteq \bigcup_{i=1}^n Z(f_i).$ Consequently
\begin{align*}
	d-1 & \stackrel{\eqref{08.02.13,15:03a}}{=} \dim (\bd A \cap Z(f)) \le \dim \left( \left( \bigcup_{i=1}^m Z(p_i) \right) \cap Z(f) \right) \\ & \stackrel{\eqref{dim:union}}{=} \max_{1 \le i \le m} \dim (Z(p_i) \cap Z(f)) \le d-1.
\end{align*}
	Hence $\dim Z(f)=d-1$ and for some $i \in \{1,\ldots,m\}$ one has $\dim (Z(p_i) \cap Z(f) ) = d-1$. Then Lemma~\ref{08.01.30,17:53} yields the assertion of Part~\ref{08.02.09,20:41A}.

	\emph{Part~\ref{08.02.09,20:46A}:} Let $a \in Z(f)$ and $\eps>0$ satisfy \eqref{08.02.11,10:04A} and \eqref{08.02.11,10:05A}. From \eqref{08.02.11,10:04A} it follows that $\dim Z(f)=d-1$. By Proposition~\ref{08.02.04,15:23}, there exists $a' \in Z(f) \cap B^d(a,\eps)$ such that $\nabla f(a') \ne o.$ We choose $\eps'>0$ such that $B^d(a',\eps') \subseteq B^d(a, \eps)$ and $\nabla f(x) \ne o$ for every $x \in B^d(a',\eps').$ Let us show that
		\begin{equation} \label{08.02.12,15:29A}
			Z(f) \cap B^d(a',\eps') \subseteq \bd A.
		\end{equation}
		 Consider an arbitrary point $x \in Z(f) \cap B^d(a',\eps').$ In view of \eqref{08.02.11,10:05A} we have $x \in A.$ On the other hand, since $f(x)=0$ and $\nabla f(x) \ne o,$ there exists a sequence $\left( x^k \right)_{k=1}^{+\infty}$ of points from $B^d(a', \eps')$ such that $f(x^k)<0$ for every $k \in \natur$ and $x^k \rightarrow x,$ as $k \rightarrow +\infty.$ Since $x^k \not\in (f)_{\ge 0}$  and $x^k \in B^d(a,\eps)$, in view of \eqref{08.02.11,10:05A} it follows that $x^k \not\in A$ for every $k \in \natur.$ Hence, $x$ is a point of $A$ and is a limit of a sequence of points lying outside $A.$ The latter implies  \eqref{08.02.12,15:29A}. Since $f(a') = 0$ and $\nabla f(x) \ne o$ for every $x \in Z(f) \cap B^d(a',\eps')$ it follows that $Z(f) \cap B^d(a',\eps')$ is an infinitely differentiable manifold of dimension $d-1,$ where the notion dimension is used in the sense of differential geometry. It is known that in the above case the Krull dimension of $Z(f) \cap B^d(a',\eps')$ is also equal to $d-1$; see \RAGbook[Proposition~2.8.14]. Consequently, we have 
		\begin{align*}
			d-1 & = \dim (Z(f) \cap B^d(a',\eps')) \stackrel{\eqref{08.02.12,15:29A}}{=} \dim (Z(f) \cap \bd A \cap B^d(a',\eps')) \\ & \le \dim (Z(f) \cap \bd A)  \le \dim (Z(f)) = d-1
		\end{align*}
		Hence $\dim (Z(f) \cap \bd A) = d-1$. By Part~\ref{08.02.09,20:41A}, it follows that $f$ coincides, up to a constant multiple, with $f_i$ for some $i \in \{1,\ldots,n\}.$ Without loss of generality we assume that $f =f_1.$ By Lemma~\ref{loc:coinc:irred}, we can choose $a'' \in Z(f) \cap B^d(a',\eps')$ such that $f_i(a'') \ne 0$ for $i \in \{2,\ldots,n\}.$ This means the sign of the polynomials $f_i, \ i=\{2,\ldots,n\},$ remains constant on $B^d(a'',\eps'').$  We prove the statement of Part~\ref{08.02.09,20:46A} by contradiction. Assume that whenever $f$ is factor of $p_i, \ i \in \{1,\ldots,m\},$ this factor is of even multiplicity. Since $\nabla f(a'') \ne o,$ we can choose $x_0, \ y_0 \in B^d(a'',\eps'')$ such that $f(x_0)>0$ and $f(y_0) <0.$ Since the signs of $f_2,\ldots,f_n$ do not change on $B^d(a'',\eps'')$ and since $f_1=f$ appears with an even multiplicity only, we obtain $\sign p_j(x_0) = \sign p_j(y_0)$ for every $j=1,\ldots,m.$ Hence $\Psi(x_0)=\Psi(y_0).$ But by \eqref{08.02.11,10:05A}, $x_0 \in A$ and $y_0 \not\in A,$ which implies that $\Psi(x_0) \ne \Psi(y_0),$ a contradiction. 
		
		The proof of Part~\ref{08.02.11,00:32A} is omitted, since it is analogous to the proof of Part~\ref{08.02.09,20:46A}.
\end{proof} 

\begin{proof}[Proof of Corollary~\ref{cor1}]  \emph{Part~\ref{08.02.09,20:47A}:} Let $b \in Z(f)$ and $\eps>0$ satisfy \eqref{08.02.11,10:06A} and \eqref{08.02.11,10:07A}. From~\eqref{08.02.11,10:06A} it follows that $\dim Z(f) =d -1.$ By Proposition~\ref{08.02.04,15:23}, there exists $b' \in Z(f) \cap B^d(b,\eps)$ such that $\nabla f(b') \ne o.$ Choose $\eps'>0$ such that $B^d(b',\eps') \subseteq B^d(b,\eps)$ and $\nabla f(x) \ne o$ for every $x \in B^d(b',\eps').$ Using arguments analogous to those from the proof of Theorem~\ref{m:thm}(\ref{08.02.09,20:46A}) we show the inclusion $Z(f) \cap B^d(b',\eps') \subseteq \bd A$ and \eqref{08.02.13,15:03a}. Hence, by Theorem~\ref{m:thm}(\ref{08.02.09,20:41A}), $f$ coincides, up to a constant multiple, with $f_i$ for some $i \in \{1,\ldots,n\}$. Without loss of generality we assume that $f=f_1.$ If $f$ is a factor of $p_i$ and $p_j$ for some $i, \, j \in \{1,\ldots,m\}$ with $i\ne j,$ we are done. We consider the opposite case, that is, for some $i \in \{1,\ldots,m\}$ the polynomial $f$ is  a factor of precisely one polynomial $p_i$ with $i \in \{1,\ldots,m\}$, say $p_1.$ We show by contradiction that in this case the factor $f$ of $p_1$ has even multiplicity. Assume the contrary, i.e., the factor $f$ of $p_1$ has odd multiplicity. Analogously to the arguments from the proof of Theorem~\ref{m:thm}, we choose $b'' \in Z(f)$ and $\eps''>0$ such that $B^d(b'',\eps'') \subseteq B^d(b',\eps')$ and $f_i(x) \ne 0$ for every $i \in \{2,\ldots,n \}$ and every $x \in B^d(b'',\eps'').$ By the choice of $b''$ and $\eps''$ we have $\sign p_i(x) = \sign p_i(b'')$ for all $i \in \{2,\ldots,m \}$ and $x \in B^d(b'',\eps'').$ Since $\nabla f(b'') \ne o$, there exist points $x_0, \, y_0 \in B^d(b'', \eps'')$ such that $f(x_0) \, f(y_0)<0.$ Then $p_1(x_0) \, p_1(y_0)<0.$ Consequently, either $p_1(x_0)>0$ or $p_1(y_0)>0.$ Without loss of generality we assume that $p_1(x_0)>0.$ It follows that $(p_i(x_0) \ge 0) \equiv (p_i(b'') \ge 0)$ for $i=1,\ldots,m.$ Hence $x_0 \in A.$ But since $f(x_0) \ne 0,$ in view of \eqref{08.02.11,10:07A}, we get $x_0 \not\in A,$ a contradiction. 

\emph{Part~\ref{08.04.13,14:27A}:} By Theorem~\ref{m:thm}(\ref{08.02.09,20:46A}) $f$ is a factor of odd multiplicity of some $p_i$ with $i \in \{1,\ldots,m\}.$ Furthermore, for some $j \in \{1,\ldots,m\}$ with $i \ne j$ the polynomial $f$ is a factor of $p_j,$ since otherwise we would get a contradiction to Part~\ref{08.02.09,20:47A}.
\end{proof} 

\begin{proof}[Proof of Corollary~\ref{m:cor:closed}] 
		By Corollary~\ref{cor1}(\ref{08.02.09,20:47A}), $f$ is a factor of some $p_i$, say $p_1.$ Without loss of generality we assume that $f_1=f.$ Let us show that the factor $f$ of $p_1$ is of even multiplicity. Assume the contrary. In view of  Proposition~\ref{08.02.04,15:23}, we can choose $a' \in \intr A \cap Z(f)$ such that $\nabla f(a') \ne o$. We fix $\eps'>0$ such that $\nabla f(x) \ne o$ for every $x \in B^d(a',\eps').$ By Lemma~\ref{loc:coinc:irred} we can choose $a'' \in B^d(a',\eps')$ such that $f_i(a'') \ne 0$ for every $i \in \{2,\ldots,n \}.$ Fix $\eps''>0$ such that for every $i \in \{2,\ldots, n \}$ the sign of $f_i$ remains constant on $B^d(a'',\eps'').$ Since $\nabla f(a'') \ne o,$ there exist $x_0$ and $y_0$ in $B^d(a'',\eps'')$ with $f(x_0) \, f(y_0) <0.$ Hence $p_1(x_0) \, p_1(y_0) < 0,$ and we get that either $x_0$ or $y_0$ does not belong to $A,$ a contradiction. 
\end{proof}

\begin{proof}[Proof of Corollary~\ref{m:cor:open}]
	Equality \eqref{08.02.11,16:45}  implies  \eqref{08.02.13,15:03a}, and hence, by Theorem~\ref{m:thm}(\ref{08.02.09,20:41A}), $f$ is a factor of $p_i$ for some $i \in \{1,\ldots,m\}.$ The rest of the proof is analogous to the proof of Corollary~\ref{m:cor:closed}.
\end{proof} 

\begin{proof}[Proof of Corollary~\ref{polyt:cor}]
	Let us prove the first part of the assertion. Assume the contrary, say $p_1$ is a factor of both $q_1$ and $q_2.$ Then within the $(d-1)$-dimensional affine space $Z(p_1)$ the facet $P \cap Z(p_1)$ of $P$ is represented by $d-2$ polynomials $q_3,\ldots,q_d$ in the following way
	$$
		P \cap Z(p_1) = \setcond{x \in Z(p_1)}{q_3(x) \ge 0, \ldots, q_d(x) \ge 0}.
	$$
	This yields a contradiction to the fact that a $k$-dimensional convex polytope cannot be represented (in the above form) by less than $k$ polynomials; see \GroetschelHenk[Corollary~2.2]. The second part of the assertion follows directly from Theorem~\ref{m:thm}(\ref{08.02.09,20:46A}).
\end{proof}

\begin{proof}[Proof of Corollary~\ref{polyg:cor}]
	For $j \in \{1,2\}$ denote by $I_j$ the set of indices $i \in \{ 1,\ldots,m \}$ for which $p_i$ is a factor of $q_j.$ By Corollary~\ref{m:cor:closed} it follows that $I_1 \cup I_2 = \{1,\ldots,m\}.$ Furthermore, $I_1 \cap I_2 = \emptyset,$ by Corollary~\ref{polyt:cor}.	Let us show that either $I_1$ or $I_2$ is empty. Assume the contrary. We show that then there exist $i \in I_1$ and $j \in I_2$ such that the edges $Z(p_i) \cap P$ and $Z(p_j) \cap P$ of $P$ are not adjacent and not parallel. Since $m \ge 7$, after possibly exchanging the roles of $q_1$ and $q_2,$ we may assume that the cardinality of $I_2$ is at least four. Let us take an arbitrary $i \in I_1.$ Then there exist at least two sides of the form $Z(p_j) \cap P, \ j \in I_2,$ which are not adjacent to $Z(p_i) \cap P.$ One of these sides is not parallel to $Z(p_i) \cap P.$  The intersection point $y$ of $Z(p_i)$ and $Z(p_j)$ lies outside $P$ and fulfills the equalities $q_1(y)=q_2(y)=0,$ a contradiction to the inclusion $(q_1,q_2)_{\ge 0} \subseteq P.$ Hence $I_1$ or $I_2$ is empty. Without loss of generality we assume that $I_2 = \emptyset.$ 

 For $i \in \{1,\ldots,m\}$ let $k_i$ be the multiplicity of the factor $p_i$ of $p_1.$ Then $q_1 = p_1^{k_1} \cdot \ldots \cdot p_m^{k_m} \, g_1$ for some polynomial $g_1,$ and Statements~\ref{08.02.14,16:37} and \ref{08.02.14,16:38} follow directly from Theorem~\ref{m:thm}(\ref{08.02.09,20:46A}). 

It remains to verify Condition~\ref{08.02.14,16:38} (which involves $g_2=q_2$). This condition can be deduced from Proposition~2.1(ii) in \GroetschelHenk, but below we also give a short proof. We argue by contradiction. Let $y$ be a vertex of $P$ with $g_2(v) > 0.$ Up to reordering the sequence $p_1,\ldots,p_m$ we may assume that $p_1(v)=0.$ Clearly, any point $y'$ lying in $Z(p_1) \setminus P$ and sufficiently close to $y$ fulfills the conditions $q_1(y')=0$ and $q_2(y') > 0.$ Hence $y' \in P,$ a contradiction to the inclusion $(q_1,q_2)_{\ge 0} \subseteq P.$
\end{proof} 

\bibliographystyle{amsplain}

\def\cprime{$'$} \def\cprime{$'$} \def\cprime{$'$} \def\cprime{$'$}
\providecommand{\bysame}{\leavevmode\hbox to3em{\hrulefill}\thinspace}
\providecommand{\MR}{\relax\ifhmode\unskip\space\fi MR }
\providecommand{\MRhref}[2]{%
  \href{http://www.ams.org/mathscinet-getitem?mr=#1}{#2}
}
\providecommand{\href}[2]{#2}

\begin{tabular}{l}
        \textsc{Gennadiy Averkov, 
	Universit\"atsplatz 2, Institute of Algebra and Geometry,
        } \\
        \textsc{Faculty of Mathematics, Otto-von-Guericke University of Magdeburg,}
	\\
	\textsc{39106 Magdeburg, Germany}
	\\
        \emph{e-mail}: \texttt{gennadiy.averkov@googlemail.com} \\
	\emph{web}: \texttt{http://fma2.math.uni-magdeburg.de/$\sim$averkov}
    \end{tabular} 

\end{document}